\documentclass{article}

\usepackage{kms}
\usepackage{lineno}

\title{\textbf{Vector bundles whose restriction to a linear section is Ulrich}}

\author{
	Rajesh S.~Kulkarni 
		\footnote{Michigan State University, East Lansing, Michigan. {\sf kulkarni@math.msu.edu}},
	Yusuf Mustopa 
		\footnote{Tufts University, Medford, Massachusetts. {\sf Yusuf.Mustopa@tufts.edu}}, and 
	Ian Shipman 
		\footnote{University of Utah, Salt Lake City, Utah. {\sf ian.shipman@gmail.com}}
}

\begin{document}
\maketitle 

\begin{abstract}
An Ulrich sheaf on an $n-$dimensional projective variety $X \subseteq \mathbb{P}^{N}$ is an initialized ACM sheaf which has the maximum possible number of global sections.   Using a construction based on the representation theory of Roby-Clifford algebras, we prove that every normal ACM variety admits a reflexive sheaf whose restriction to a general 1-dimensional linear section is Ulrich; we call such sheaves $\delta-$Ulrich.  In the case $n=2,$ where $\delta-$Ulrich sheaves satisfy the property that their direct image under a general finite linear projection to $\P^2$ is a semistable instanton bundle on $\mathbb{P}^{2}$, we show that some high Veronese embedding of $X$ admits a $\delta-$Ulrich sheaf with a global section. 
\end{abstract}

\section*{Introduction}
The structure theory of ACM sheaves on a subvariety $X \subseteq \mathbb{P}^{N}$ is an important and actively studied area of algebraic geometry.  Ulrich sheaves are the ``nicest possible" ACM sheaves on $X,$ since their associated Cohen-Macaulay module has the maximum possible number of generators, they are closed under extensions (they form an abelian subcategory of ${\rm Coh}(X)$), and their Hilbert series is completely determined by their rank and $\deg(X).$  Moreover, they are all Gieseker-semistable.  

Ulrich sheaves on $X$ admit a clean geometric characterization; they are precisely the sheaves whose direct image under a general finite linear projection $\pi : X \to \P^{n}$ is a trivial vector bundle.  In particular, a sheaf on $\P^n$ is Ulrich with respect to $\cO_{\P^{n}}(1)$ if and only if it is a trivial vector bundle.  Varieties known to admit Ulrich sheaves include curves and Veronese varieties \cite{ESW,H}, complete intersections \cite{HUB}, generic linear determinantal varieties \cite{BHU}, Segre varieties \cite{CMP}, rational normal scrolls \cite{MR}, Grassmannians \cite{CMR}, some flag varieties \cite{CMR,CHW}, generic K3 surfaces \cite{AFO}, abelian surfaces \cite{Bea}, Enriques surfaces \cite{BN} and ruled surfaces \cite{ACM}.  

The question of whether every subvariety of projective space admits an Ulrich sheaf was first posed in \cite{ESW} and remains open.  It was shown in \cite{kms-1} that an affirmative answer is equivalent to the simultaneous solution of a large number of higher-rank Brill-Noether problems on nongeneric curves.  In light of the fact that the varieties known to admit Ulrich sheaves are mostly ACM, a natural first step is to restrict the question to ACM varieties.  

It is straightforward to check that if $\cE$ is an Ulrich sheaf on $X,$ then the restriction of $\cE$ to a general linear section is Ulrich.  The converse holds for linear sections of dimension 2 or greater (Lemma \ref{lem:rest-lin-sec}) but not linear sections of dimension 1 (e.g. Remark \ref{rmk:lin-det-examples}).  In addition, Ulrich sheaves on 1-dimensional linear sections have very recently been used by Faenzi and Pons-Llopis to show that most ACM varieties are of wild representation type \cite{FaPL}.  All this suggests a natural enlargement of the class of Ulrich sheaves whose existence problem may be more tractable.  

\begin{definition*}
Let $\cE$ be a reflexive sheaf on a projective variety $X \subseteq \P^N,$ and let $\cO(1) = \cO_{\P^{N}}(1)|_{X}.$  We say that $\cE$ is \emph{$\delta-$Ulrich} (with respect to $\cO(1)$) if there exists a smooth 1-dimensional linear section $Y$ of $X$ such that the restriction $\cE|_{Y}$ is Ulrich with respect to $\cO_{Y}(1).$ 
\end{definition*}

Perhaps the simplest examples of $\delta-$Ulrich sheaves that are not Ulrich are constructed in Theorem 2.2.5 of \cite{OSS}---they are nontrivial rank-2 bundles on $\P^2$ whose restriction to a general line is trivial.  By definition they are $\delta-$Ulrich with respect to $\cO_{\P^2}(1),$ but their nontriviality means they are not Ulrich.  Though we are not aware of a complete characterization of sheaves on $\P^2$ that are $\delta-$Ulrich with respect to $\cO_{\P^{2}}(1),$ Theorem 17 of \cite{Jar} implies they are all semistable instanton sheaves in the sense of \cite{Jar}.  The connection with instanton sheaves is discussed in Section \ref{subsec:instanton}.

While $\delta-$Ulrich sheaves are a strict generalization of Ulrich sheaves, their existence is not immediate; it is worth mentioning that the examples we present in this paper (specifically, in Remarks \ref{rmk:lin-det-examples} and \ref{rmk:non-loc-free}) are on varieties already known to admit Ulrich sheaves.  Our main result, which is implied by Theorem \ref{thm:reflexive-delta-ulrich}, is the following:  

\begin{customthm}{A}\label{main-theorem}
Let $X \subseteq \mathbb{P}^{N}$ be a normal ACM variety.  Then $X$ admits a $\delta-$Ulrich sheaf.
\end{customthm}

The $\delta-$Ulrich condition for a sheaf $\cF$ on $X$ can be rephrased as saying that if $\pi : X \to \mathbb{P}^{n}$ is a general finite linear projection, the direct image $\pi_{\ast}\cF$ restricts to a trivial vector bundle on a general line $\ell \subseteq \mathbb{P}^{n},$ so to construct a $\delta-$Ulrich sheaf on $X$ amounts to finding a reflexive sheaf $\cE$ on $\mathbb{P}^{n}$ and a line $\ell \subseteq \mathbb{P}^{n}$ such that $\cE$ is a $\pi_{\ast}\cO_{X}-$module and $\cE|_{\ell}$ is a trivial vector bundle on $\ell.$  It suffices to carry this out this construction on an open subset of $\mathbb{P}^{n}$ whose complement is of codimension at least 2.

Lemma \ref{lem:monogenic-opens} implies that if $\pi : X \to \mathbb{P}^{n}$ is a finite linear projection, then there are open affine subsets $V_{1},V_{2} \subseteq \mathbb{P}^{n}$ and polynomials $p_{i}(z_{i}) \in \cO_{V_i}[z_i]$ such that the complement of $V_{1} \cup V_{2}$ is of codimension 2 and $\pi_{\ast}\cO_{X}|_{V_i} \cong \cO_{V_{i}}[z_i]/(p_{i}(z_i)).$  Our strategy for proving Theorem A begins with constructing for $i=1,2$ a locally Cohen-Macaulay sheaf $\cE_i$ on $V_i$ which admits the structure of a $\pi_{\ast}\cO_{X}|_{V_i}$-module.  What allows us to do this is the notion of a \textit{characteristic morphism} of (sheaves of) algebras.  Such morphisms generalize algebra homomorphisms in the sense that they respect the Cayley-Hamilton theorem; see Section \ref{subsec:roby-char} for details, as well as \cite{kms-2}.  Although we are not aware of any earlier work on characteristic morphisms as such, we were inspired by the use of characteristic polynomials in \cite{Pa}.  For similar ideas in the context of invariant theory, see \cite{Pr}.   

It is not obvious that $\cE_1$ and $\cE_2$ glue together to form a $\pi_{\ast}\cO_{X}|_{V_{1} \cup V_{2}}-$module.  However, the special characteristic morphism we construct in Proposition \ref{prop:acm-fil-ps} ensures that the restrictions of $\cE_1$ and $\cE_2$ to a general line $\ell \subseteq V_1 \cup V_2$ glue together to form an Ulrich sheaf for the restriction $\pi^{-1}(\ell) \to \ell$ of $\pi.$  The $\delta-$Ulrich sheaf we produce is an algebraization of a sheaf on the formal neighborhood of $\ell$ which comes from gluing completions of $\cE_1$ and $\cE_2$ along this neighborhood (Lemma \ref{lem:algebrize} and Theorem \ref{thm:reflexive-delta-ulrich}).

Even though it is not used explicitly, the central concept underlying the proof of Proposition \ref{prop:acm-fil-ps} is that of the Roby-Clifford algebra $R_F$ of a degree-$d$ homogeneous form $F$ over a field $\bk.$  This was introduced by Roby in \cite{Ro}, and it directly generalizes the classical Clifford algebra of a quadratic form, as $R_F$ satisfies a similar, higher-degree universal property (see Remark \ref{rmk:def-roby-clifford}).  It is shown in \cite{VDB} that Ulrich sheaves on the cyclic covering hypersurface $\{w^{d} = F\}$ correspond to finite-dimensional $R_{F}-$modules, and a more refined correspondence involving the natural $\Z/d\Z$-grading on $R_F$ is used in \cite{BHS} to construct Ulrich sheaves on hypersurfaces.  The latter construction uses the \textit{$\Z/d\Z$-graded tensor product} of modules over Roby-Clifford algebras (see Section \ref{prelim:subsec:twisted-tensor}) to construct an Ulrich sheaf over the zero locus of the ``generic homogeneous form of degree $d$ which is a sum of $s$ monomials."  Our proof of Proposition \ref{prop:acm-fil-ps} uses $\Z/d\Z$-graded tensor products to extend an algebraic object (the characteristic morphism) from the line $\ell \subset \P^n$ to all of $\P^n$. 

We can say more about $\delta-$Ulrich sheaves when $X$ is a normal ACM surface.  It is immediate from the definition that $\delta-$Ulrich sheaves on normal ACM surfaces are locally Cohen-Macaulay, a necessary condition for being Ulrich.  As mentioned earlier, the sheaves on $\P^2$ which are $\delta-$Ulrich with respect to $\cO_{\mathbb{P}^{2}}(1)$ are semistable instanton sheaves, so in general, $\delta-$Ulrich sheaves on a surface have the property that their direct image under a finite linear projection is a semistable instanton sheaf (Proposition \ref{prop:surface-instanton}).  We show that the intermediate cohomology module $H^{1}_{\ast}(\cE)$ satisfies the Weak Lefschetz property (Proposition \ref{prop:coh-mod-WLP}); moreover, the maximum value of the Hilbert function of $H^{1}_{\ast}(\cE)$ is $h^{1}(\cE(-1))$.   

A substantial difference between Ulrich and $\delta-$Ulrich sheaves is that the former are globally generated, while the latter need not have any global sections at all (compare Remark \ref{rmk:lin-det-examples}).  However, a $\delta-$Ulrich sheaf $\cE$ on $X$ is Ulrich if and only if it has ${\rm deg}(X) \cdot {\rm rk}(\cE)$ global sections (see Proposition \ref{prop:ulrich-equiv-conditions}).  If we replace $\cO_{X}(1)$ by a potentially high twist, we have enough control on the cohomology to obtain the following result.

\begin{customthm}{B}
If $X \subseteq \mathbb{P}^{N}$ is a smooth ACM surface, there exists $k > 0$ such $X$ admits a $\delta-$Ulrich sheaf with respect to $\cO_{X}(k)$ possessing a global section.
\end{customthm}

This theorem follows from a more precise statement.  If $\cE$ is a $\delta-$Ulrich sheaf on $X,$ consider the quantity
\[ \alpha(\cE) = h^0(\cE)/\deg(X)\rk(\cE) \]
Our earlier observation can be rephrased as saying that $\cE$ is Ulrich if and only if $\alpha(\cE) = 1$.  Theorem B is proved by exhibiting a sequence of sheaves $\{\cE_m\}_{m}$ where $\cE_m$ is $\delta$-Ulrich with respect to $\cO_X(2^m)$ and such that $\lim_{m \to \infty}{ \alpha(\cE_m) } = 1$.

\subsection*{Acknowledgments}

I.S. was partially supported during the preparation of this paper by National Science Foundation award DMS-1204733. R. K. was partially supported by the National Science Foundation awards DMS-1004306 and DMS-1305377.  We would like to thank the referee for helpful comments.

\subsection*{Notation and Conventions}

Our base field $\bk$ is algebraically closed of characteristic zero.  All open subsets are Zariski-open.  If $R$ is a ring we use the notation $R\{t_1,\dotsc,t_n\}$ for the free $R$-module with basis $t_1,\dotsc,t_n$.


\section{Preliminaries} \label{sec:preliminaries}
In this section, we collect the algebraic prerequisites for the proof of Theorem \ref{main-theorem}.  Throughout, $R$ denotes a commutative $\bk-$algebra and $A$ denotes a commutative $R-$algebra which is free of rank $d \geq 2$ as an $R-$module.

\subsection{Roby Modules and Characteristic Morphisms}
\label{subsec:roby-char}

\begin{definition}
Let $M,W$ be free $R-$modules and let $F \in {\rm Sym}^{\bullet}_{R}(M^{\vee})$ be a homogeneous form of degree $e \geq 2.$  An $R-$module morphism $\phi : M \to \End_R(W)$ is an $F-$Roby module if for all $m \in M$ we have 
\[ \phi(m)^{e} = F(m) \cdot \id_W \]
where $F(m)$ is the image of $m^{\otimes e}$ under the symmetric map $M^{\otimes e} \to R$ associated to $F.$  If, in addition, $W$ is a $\mathbb{Z}/e\mathbb{Z}$-graded $R-$module and $\phi(m)$ is a degree-1 endomorphism for $0 \neq m \in M,$ we say that $\phi$ is a \textit{graded} $F-$Roby module.
\end{definition}

\begin{rmk}
\label{rmk:def-roby-clifford}
The terminology can be explained as follows.  If $\phi$ is an $F-$Roby module, the induced $R-$algebra morphism $T^{\bullet}_{R}(M) \to {\rm End}_{R}(W)$ annihilates $\{\phi(m)^{e}-F(m) : m \in M\},$ and therefore descends to a morphism $R_{F} \to {\rm End}_{R}(W),$ where 
\[ R_{F} := T^{\bullet}_{R}(M)/{\langle}\phi(m)^{e}-F(m) : m \in M{\rangle} \]
is the Roby-Clifford algebra of $F$ (see \cite{Ro}).  Conversely, given an $R-$algebra morphism $R_{F} \to {\rm End}_{R}(W),$ we recover an $F-$Roby module by composing with the natural injection $M \hookrightarrow R_{F}.$
\end{rmk}

\begin{example}\label{ex:monomial-Roby-module}
We recall a construction from \cite{Ch} which will be used in the proof of Proposition \ref{prop:acm-fil-ps}.  Let $M = R\{x_1,\dotsc,x_n\}$ and suppose that $y_1,\dotsc,y_n$ is the dual basis of $M^\vee$.  Consider a monomial $F=y_{i_1}y_{i_2} \dotsc y_{i_e} \in \Sym_R^e( M^\vee)$ and put $W = R\{w_1,\dotsc,w_e\}$.  Then there is a natural, $\Z/e\Z-$graded $F-$Roby module $\phi:M \to \End_R(W)$ given by
\[ \phi( x_i )(w_j) = \begin{cases} w_{j+1} & i = i_j, \\ 0 & \text{otherwise}, \end{cases} \]
where the indices on the elements $w_1,\dotsc,w_e$ are taken modulo $e,$ and $\deg(w_i) = i$ for all $i$.
\end{example}

\begin{definition}
If $A$ is an associative $R-$algebra whose underlying $R-$module is of finite rank $d,$ the \textit{characteristic polynomial} of $A$ is
\[
\chi_{A}(t,a):={\rm det}(tI-\rho_{A}(a)) = \sum_{j=0}^{d}(-1)^{j}{\rm tr}(\wedge^{j}\rho_{A}(a)) \cdot t^{d-j} 
\]
where $\rho_{A} : A \to {\rm End}_{R}(A)$ is the regular representation of $A.$
\end{definition}

Observe that $\chi_{A}(t,a)$ is a degree-$d$ element of ${\rm Sym}^{\bullet}_{R}(A^{\vee}) \otimes_R R[t] \cong {\rm Sym}^{\bullet}_{R}(A^{\vee} \oplus R{\{}t{\}}).$  Also, if $B$ is an $R-$algebra, then for any $a \in A$ and $b \in B$ we have that $\chi_{A}(b,a)$ is a well-defined element of $B.$  

\begin{example}\label{ex:totally-split}
Consider the $R-$algebra $A=R^{\times d}$.  We identify $R^{\times d} = R\{e_1,\dotsc,e_d\}$ where $\{e_i\}$ is the standard basis of idempotents.  Under the regular representation we have $\rho_A(a_1,\dotsc,a_d) = \operatorname{diag}(a_1,\dotsc,a_d)$ and therefore $\chi_A(t,a_1,\dotsc,a_d) = (t-a_1)\dotsm(t-a_d)$.  It folows that 
\[ \chi_A(t) = (t-x_1)\dotsm (t-x_d) \]
where $x_1,\dotsc,x_d$ is the dual basis to $e_1,\dotsc,e_d$.
\end{example}

We record the following elementary properties, which will be used in the sequel.

\begin{lem}
\label{lem:char-base-change}
Let $B$ be an $R-$algebra.
\mbox{}
\begin{itemize}
\item[(i)]{ If $B$ is commutative and free of finite rank as an $R-$module, then $\chi_{A}$ is taken to $\chi_{A \otimes_{R} B}$ under the natural map 
\[
{\rm Sym}^{\bullet}_{R}(A^{\vee})[t] \to {\rm Sym}^{\bullet}_{B}((A \otimes_{R} B)^{\vee})[t]
\]
induced by the base-change map $A^{\vee} \to (A \otimes_{R} B)^{\vee} = {\rm Hom}_{B}(A \otimes_{R} B, B).$}
\item[(ii)]{If $B \to C$ is an embedding of $R-$algebras which are both free of the same finite rank, then $\chi_B$ is the image of $\chi_C$ under the natural morphism
\[ \Sym^\bt_R(C^\vee)[t] \to \Sym^\bt_R(B^\vee)[t]. \]   \hfill \qedsymbol }
\end{itemize} 
\end{lem}

If $\phi : A \to B$ is a morphism of $R-$algebras, the Cayley-Hamilton theorem implies that 
\[ \chi_{A}(\phi(a),a) = \phi(\chi_{A}(a,a))=0 \]
for all $a \in A.$  The more general notion that follows is a key ingredient in our construction of $\delta-$Ulrich sheaves.

\begin{definition}
If $B$ is an $R-$algebra, an $R-$module morphism $\phi:  A \to B$ is a \textit{characteristic morphism} if $\chi_{A}(\phi(a),a)=0$ for all $a \in A.$
\end{definition}

\begin{rmk}
The notion of a characteristic morphism is strictly more general than that of an $R-$algebra morphism.  If $A = R{\{}e_{1},e_{2}{\}}$ is the $R-$algebra generated by the orthogonal idempotents $e_{1}$ and $e_{2},$ then for any $a,b \in R$ satisfying $a+b \neq 0,$ the map $\phi : A \to {\rm Mat}_{2}(R)$ defined by 
\[
\phi(e_{1})=\begin{pmatrix} 1 & a \\ 0 & 0 \end{pmatrix}, \hskip5pt \phi(e_{2}) = \begin{pmatrix} 0 & b \\ 0 & 1 \end{pmatrix}
\]
is a characteristic morphism, but not an $R-$algebra morphism.
\end{rmk}

We now turn to the sheaf-theoretic formulations of these concepts.  For the remainder of this subsection, $Y$ denotes a smooth irreducible quasi-projective variety, $\cA$ denotes a sheaf of $\cO_{Y}-$algebras which is locally free of rank $d \geq 2,$ and $W$ denotes a finite-dimensional $\bk$-vector space. For a sheaf $\cF$ on $Y$, we denote the stalk of $\cF$ at a point $y \in Y$ by $\cF_y$ and the ${\bk}(Y)-$vector space of rational sections of $\cF$ by $\cF(Y).$

\begin{definition}
\label{defn:char-mor}
If $\cB$ is a coherent sheaf of $\cO_{Y}-$algebras, a $\cO_{Y}-$linear morphism $\phi : \cA \to \cB$ is a \textit{characteristic morphism} if for each $y \in Y$, the $\cO_{Y,y}-$module morphism $\phi_y : \cA_y \to \cB_y$ is a characteristic morphism.
\end{definition} 

The following observation will be used later.

\begin{lem}
\label{lem:char-mor-field}
$\phi : \cA \to \cB$ is a characteristic morphism if and only if the induced ${\bk}(Y)-$linear map $\phi_{{\bk}(Y)} : \cA(Y) \to \cB(Y)$ is a characteristic morphism. \hfill \qedsymbol
\end{lem}

If $\rho_{A} : \cA \to {\cE}nd(\cA)$ is the regular representation of $\cA,$ then since ${\rm tr}(\wedge^{j}\rho_{\cA})$ is a global section of ${\rm Sym}^{j}(\cA^{\vee})$ for each $j,$ there exists a global characteristic polynomial $\chi_{\cA} \in H^{0}({\rm Sym}^{\bullet}(\cA^{\vee} \oplus \cO_{Y}\{t\})).$  

\begin{definition}
An $\cO_{Y}-$linear morphism $\psi : \cA \oplus \cO_{Y}\{T\} \to {\rm End}(W \otimes \cO_{Y})$ is a \emph{$\chi_{\cA}-$Roby module} if for each $y \in Y$, the $\cO_{Y,y}-$module morphism $\psi_y : \cA_y \oplus \cO_{Y,y}\{T\} \to {\rm End}(W \otimes \cO_{Y,y})$ is a $\chi_{\cA}-$Roby module in the sense that for all $a \in \cA_y$ and all $r \in \cO_{Y,y}$ we have
\[ \psi(a,rT)^{d} = \chi_{\cA}(a,r) \cdot {\rm Id}. \]
If $W$ is $\Z/d\Z-$graded and $\psi(a,rT)$ is a degree-1 endomorphism for all local sections $a,r$ then $\psi$ is a \emph{graded $\chi_\cA-$Roby module}.
\end{definition}

If $\psi$ is a graded $\chi_{\cA}$-Roby module as above, then $\psi(T)$ is globally defined, and $\psi(T)^{d} = {\rm Id}$ since $\chi_\cA$ is monic in $t$.  In particular, $\psi(T)$ is invertible in ${\rm End}(W \otimes \cO_{Y}).$

\begin{lem}
\label{lem:t-inverted}
Let $\psi : \cA \oplus \cO_{Y}\{T\} \to {\rm End}(W \otimes \cO_{Y})$ be a graded Roby $\chi_{\cA}-$module.  Then the morphism $C_{\psi} : \cA \to {\rm End}(W \otimes \cO_{Y})$ defined by the composition
\[
\cA \into \cA \oplus \cO_{Y}\{T\} \xrightarrow{-\psi} {\rm End}(W \otimes \cO_{Y}) \xrightarrow{\cdot \psi(T)^{-1}} {\rm End}(W \otimes \cO_{Y})
\]
is a characteristic morphism.
\end{lem}

\begin{proof}
By Lemma \ref{lem:char-mor-field} it suffices to consider a field extension $K/\bk$ and a $d-$dimensional commutative $K-$algebra $A$ in place of $\cO_{Y}$ and $\cA,$ respectively.  Put $W_K = W \tensor_\bk K$.  Let $\chi_{A}=\chi_{A}(t)$ be the characteristic polynomial of $A,$ and let $\psi:  A \oplus K\{T\} \to {\rm End}_{K}(W_K)$ be a graded $\chi_{A}-$Roby module.  Then $\psi$ corresponds to an element $\psi^{\vee}$ of ${\rm End}_{K}(W_K) \otimes (A^{\vee} \oplus K\{t\})$ whose $d-$th power 
\[
(\psi^{\vee})^{d} \in {\rm End}_{K}(W_K) \otimes {\rm Sym}^{d}(A^{\vee} \oplus K\{t\}) \cong {\rm Hom}(W_K,W_K \otimes  {\rm Sym}^{d}(A^{\vee} \oplus K\{t\}))
\] 
is equal to $1_{W} \otimes \chi_{A}.$  

Consider the graded $S = \Sym^\bt(A^\vee)[t,w]/(w^d-\chi_A)$-module $M = W_K \tensor_K \Sym^\bt(A^\vee)[t]$ on which $w$ acts by $\psi^\vee$ (and $A^\vee,t$ act in the obvious way).  Now, $M$ is a graded maximal Cohen-Macaulay $S$-module, generated in degree zero--a graded Ulrich module, in fact.  Since $\Sym^\bt(A^\vee)[w]$ is a standard-graded polynomial subring of $S$ over which $S$ is finite and flat, it follows that $M$ is graded-free over $\Sym^\bt(A^\vee)[w]$ and generated in degree zero.  Consequently, the map 
\[ W_K \tensor_K \Sym^\bt(A^\vee)[w] \to M \]
is an isomorphism.  We aim to compute the action of $t$ in terms of the action of $w$ and $A^\vee$.  We can write 
\[ \psi^\vee = \psi_0^\vee + \psi(T) \tensor t, \quad \psi_0^\vee \in \End_K(W_K)\tensor A^\vee.\]
So if $m \in M$ we have
\[ w m = \psi_0^\vee m + t \psi(T)m. \]
As we observed earlier, $\psi(T)$ is invertible.  Replacing $m$ by $\psi(T)^{-1}m$ in the previous equation, we have that
\[  tm = w \psi(T)^{-1} m - \psi_0^\vee \psi(T)^{-1} m. \]
Let $\overline{M}$ be the graded $\Sym^\bt(A^\vee)[t]/(\chi_A(t))$-module obtained from reducing the $S-$module structure of $M$ modulo $w$.  Then $\overline{M}$ is graded-free over $\Sym^\bt(A^\vee),$ and it is generated in degree zero.  Now, $t$ acts on $\overline{M}$ by $-\psi^\vee_0 \psi(T)^{-1}$.  Since $\chi_A(t)$ acts by zero on $\overline{M}$, the map $A \to \End_K(W_K)$ corresponding to $-\psi^\vee_0 \psi(T)^{-1}$ is a characteristic morphism.  This map is $C_\psi$, so we see that $C_\psi$ is a characteristic morphism.
\end{proof}

\begin{example}\label{ex:totally-split-2}
The following construction will be used in the proof of Lemma \ref{lem:integral-curve-map}.  Once again, we consider $A = R^{\times d} = R\{e_1,\dotsc,e_d\}$.  From Example \ref{ex:totally-split} we see that $\chi_A(t) = \prod_{i=1}^d{(t-x_i)}$ where $\{x_i\}$ is the dual basis to $\{e_i\}$.  There is a natural graded $\chi_A-$Roby module $\phi:A\oplus R\{T\} \to \End_R( R\{w_1,\dotsc,w_d\} )$ defined by
\[ \phi(T)(w_i) = w_{i+1}, \quad \phi(e_i)(w_j) = \begin{cases} -w_{j+1} & i = j+1, \\ 0 & i \neq j+1, \end{cases} \]
where the indices on the $w_i$ are taken modulo $d$ and $\deg(w_i)=i$.  Since
\[ \phi(r T + \sum_{i=1}^r{a_i e_i})(w_j) = (r-a_{j+1})w_{j+1} \]
we see by iteration that $\phi$ is indeed a $\chi_A-$Roby module.  Let us compute $C_\phi$.  We have
\[ C_\phi(e_i)(w_j) = -\phi(e_i)(\phi(T)^{-1}(w_j)) = -\phi(e_i)(w_{j-1}) = \delta^i_j w_j. \]
So, better than simply being a characteristic morphism we see that $C_\phi$ is an algebra morphism, equipping $W$ with the structure of a free $R^{\times d}$-module.  However, note that the construction of this module implicitly relied on a cyclic ordering on the idempotents $e_\bt \in A$.
\end{example}

\begin{rmk}
We note that the formation of $C_\phi$ is functorial.  More precisely, suppose that $\phi:\cA \oplus \cO\{T\} \to \End(W)\tensor \cO$ is a $\chi$-Roby morphism.  If $W' \subset W$ is an invariant subspace in the sense that for any local section $a + r T$ of $\cA \oplus \cO\{T\}$, the action of $\phi(a,rT)$ on $W \tensor \cO$ sends $W' \tensor \cO$ into itself, then the action of $\cA$ via $C_\phi$ will also send $W' \tensor \cO$ into itself.  This equips $W' \tensor \cO$ and $W/W' \tensor \cO$ with the structures of $\chi$-Roby modules and characteristic modules, respectively.
\end{rmk}

\subsection{$\Z/d\Z-$Graded Tensor Products}
\label{prelim:subsec:twisted-tensor}

The following notion, which was first applied to the study of Roby modules in \cite{Ch}, is required for the proof of Proposition \ref{prop:acm-fil-ps}.  The proof is essentially that of Theorem 3.1 in \cite{BHS}. 

\begin{prop-def}
\label{prop:twisted-tensor-roby}
Let $M$ be a free $R-$module and $F_1,F_2 \in \Sym^d_R(M^\vee)$ homogeneous forms.  Suppose that for each $i=1,2$ we have a graded $F_i-$Roby module $\phi_i:M \to \End_R(W_i)$, where $W_1,W_2$ are $\Z/d\Z-$graded $R-$modules.  Then the morphism $\phi:M \to \End_R( W_1 \tensor W_2 )$ defined by
\[ \phi(m)( w_1 \tensor w_2 ) = \phi_1(m)(w_1) \tensor w_2 + \xi^{\deg(w_1)}w_1 \tensor \phi_2(m)(w_2) \]
is a graded $F_1+F_2-$Roby module, where $W_1 \tensor W_2$ is graded by $\deg(w_1 \tensor w_2) = \deg(w_1)+\deg(w_2)$ for homogeneous elements $w_i \in W_i$.  We denote this morphism by $\phi = \phi_1 \wh{\tensor}_\xi \phi_2$.
\end{prop-def}


\section{Construction of $\delta$-Ulrich Sheaves} \label{sec:construction}
We now take up the proof of Theorem \ref{main-theorem} in earnest.  As our first step, we use Lemma \ref{lem:t-inverted} to produce a type of enhanced Ulrich sheaf for any finite covering of $\P^1.$

\begin{lem}
\label{lem:integral-curve-map}
Let $C$ be a smooth curve and let $f : C \to \mathbb{P}^{1}$ be a morphism of degree $d \geq 2$  Then there exists a graded $\chi_{f_{\ast}\cO_{C}}$-Roby module $\psi$ whose associated characteristic morphism $C_{\psi}$ is an $f_{\ast}\cO_{C}-$module morphism.
\end{lem}

\begin{proof}
Let $K(C)/K(\P^1)$ be the field extension corresponding to $f$.  Since the extension is separated, it has the form $K(C)\cong K(\P^1)[z]/(p(z))$ for some polynomial $p(z)$.  Let $L$ be the splitting field of $p(z)$ and $g:D \to \P^1$ the map of curves corresponding to $L/K(\P^1)$.  Then $C \times_{\P^1} D$ has $d$ components, each of which is isomorphic to $D$.  So we have a diagram
\[
\xymatrix@=8pt{ \sqcup_{i=1}^d D \ar[dr]^{\eta} & & \\ & C \times_{\P^1} D \ar[r] \ar[d] & D \ar[d]^{g} \\ & C \ar[r]^{f} & \P^1 }
\]
where $\eta$ is the normalization of $C \times_{\P^1} D$.  Let $\cA = f_*\cO_C$, $\cB = g_*\cO_D$ and $\cO = \cO_{\mathbb{P}^{1}}.$  Since $C$ and $D$ are reduced curves, they are locally CM, so $\cA$ and $\cB$ are locally free as $\cO-$modules; in particular, $\cA \otimes_{\cO} \cB$ is locally free as a $\cB-$module.  Also, $\eta$ is induced by a $\cB-$module morphism $\widetilde{\eta} : \cA \tensor_\cO \cB \to \bk^{\times d} \otimes_{\bk} \cB$.  

Let $\chi(t) \in \Sym^\bt_\cO(\cA^\vee)[t]$, $\wt{\chi}(t) \in \Sym^\bt_\cB( (\cA \tensor_\cO \cB)^\vee)[t]$, and $\wt{\chi}_s(t) \in \Sym^\bt_\cB( (\cB^{\vee} \tensor_\bk \bk^{\times d})^\vee)[t]$ be the characteristic polynomials of $\cA$ over $\cO$, $\cA \tensor_\cO \cB$ over $\cB$, and $\cB \tensor_\bk \bk^{\times d}$ over $\cB$ respectively.  Then (according to Lemma \ref{lem:char-base-change}) under the natural maps
\begin{align*} \Sym^\bt_\cB( (\cB \tensor_\bk \bk^{\times d} )^\vee)[t] \to \Sym^\bt_\cB((\cA \tensor_\cO \cB)^\vee)[t] \\ \Sym^\bt_\cO(\cA^\vee)[t] \to \Sym^\bt_\cB( (\cA \tensor_\cO \cB)^\vee)[t] 
\end{align*}
we see that $\wt{\chi}_s(t)$ maps to $\wt{\chi}(t)$ and $\chi(t)$ maps to $\wt{\chi}(t)$.  This means that if $a$ and $r$ are local sections of $\cA$ and $\cO$, respectively, then $\wt{\chi}_s(a,rt) = \chi(a,rt)$.

As in Example \ref{ex:totally-split-2}, there is a natural graded $\wt{\chi}_s(t)$-Roby module  $\phi :(\bk^{\times d} \tensor_\bk \cB) \oplus \cB\{T\} \to \End_\cB(\cB \tensor_\bk \bk^{\times d})$ defined by
\[ \phi(T)(e_j) = e_{j+1}, \quad \psi( e_i )( e_j ) = \begin{cases} -e_{j+1} & i = j+1 \\ 0 & i \neq j+1 \end{cases} \]
where the indices of the standard idempotents $e_i$ are taken modulo $d$.  Clearly, $\phi$ is a Roby module for the characteristic polynomial $\wt{\chi}$ for $\cB \tensor_\bk \bk^{\times d}$ over $\cB$. Moreover, one can verify that $C_\phi$ is an algebra morphism.

By \cite{ESW}, there is an Ulrich sheaf $\cE$ for $D$ over $\P^1$, which we view as a $\cB$-module on $\P^1$.  Note that there is an algebra morphism
\[ \End_\cB(\cB \tensor \bk^{\times d}) \to \End_\cO(\cE \tensor \bk^{\times d}) \]  
defined by tensoring a map with $\cE$ over $\cB$.  We can then obtain a map
\[ \wt{\phi}: \cB \tensor_\bk \bk^{\times d} \oplus \cB\cdot T \to \End(\cB \tensor_\bk \bk^{\times d}) \to \End_\cO(\cE \tensor_\bk \bk^{\times d}) = \End(W \tensor \bk^{\times d} \tensor \cO) \]
where we choose a trivialization $\cE \cong W \tensor \cO$ as $\cO$-modules.  Now this induces a $\chi$-Roby module structure on $W \tensor \bk^{\times d} \tensor \cO$.  Let $\psi:\cA \oplus \cO\{T\} \to \End(W\tensor \bk^{\times d})\tensor \cO$ be the restriction of $\wt{\phi}$ to $\cA \oplus \cO\{T\} \subset \cB \tensor \bk^{\times d} \oplus \cB\{T\}$, where $\cA \to \cB \tensor \bk^{\times d}$ is the map $\cA \tensor 1 \to \cA \tensor \cB \xrightarrow{\wt{\eta}} \cB \tensor \bk^{\times d}$. 

Suppose that $a,r$ are local sections of $\cA$ and $\cO$ respectively.  Then we know that 
\[\phi(a,r)^d = \wt{\chi}_s(a,r)\cdot \id \in \End_\cB(\cB \tensor \bk^{\times d})\]  
Now $\wt{\chi}_s(a,r) = \chi(a,r) \in \cO$.  Since the morphism $\End_\cB(\cB \tensor \bk^{\times d}) \to \End(W \tensor \bk^{\times d}) \tensor \cO$ is $\cO$-linear, $\psi$ is a $\chi$-Roby module.  Finally, we must check that $C_\psi$ is a morphism.  Now, $C_\phi = R_{\phi(T)^{-1}} \phi$ is a morphism, where $R_{\phi(T)^{-1}}$ is right multiplication by $\phi(T)^{-1}$.  Composing $C_\phi$ with the map $\End_\cB(\cB \tensor \bk^{\times d}) \to \End(W \tensor \bk^{\times d}) \tensor \cO$, we obtain $R_{\wt{\phi}(T)^{-1}} \wt{\phi}$ and this is still a morphism.  Now, if we restrict this morphism to $\cA$ we obtain $C_\psi = R_{\psi(T)^{-1}} \psi$ since $\psi(T) = \wt{\phi}(T)$ and $\psi$ is the restriction of $\wt{\phi}$ to $\cA$.
\end{proof}

\begin{definition}
Let $W$ be a vector space, and let $F^{\bullet}$ be an increasing filtration on $W.$  A \textit{filtered pseudomorphism} $\phi: \cA \to {\rm End}(W \otimes \cO_{\mathbb{P}^{n}})$ is a characteristic morphism satisfying the following properties:
\begin{itemize}
\item[(i)]{The image of $\phi$ is contained in the algebra ${\rm End}_{F^{\bullet}}(W \otimes \cO_{\mathbb{P}^{n}})$ of endomorphisms preserving $F^{\bullet}.$}
\item[(ii)]{The induced map $\phi_{F^{\bullet}} : \cA \to \Pi_{i}{\rm End}(F^{i+1}W/F^{i}W \otimes \cO_{\mathbb{P}^{n}})$ is an $\cO_{\mathbb{P}^{n}}$-algebra morphism.}
\end{itemize}
\end{definition}

Our $\delta-$Ulrich sheaf will come from a characteristic morphism that restricts to a filtered pseudomorphism on a 1-dimensional linear section.

\begin{definition}
If $R$ is a commutative ring, a finitely generated $R-$algebra $A$ is said to be \textit{monogenic} if there exists a monic polynomial $p[z] \in R[z]$ such that $A \cong R[z]/{\langle}p(z){\rangle}.$  If $Y$ is a quasi-projective variety and $U \subseteq Y$ is open, a coherent sheaf $\cA$ of $\cO_Y$-algebras is said to be \textit{monogenic on} $U$ if $\cA|_U$ is a monogenic $\cO_U-$algebra.
\end{definition}  

\begin{lem}
\label{lem:monogenic-opens}
Let $X \subseteq \mathbb{P}^{N}$ be a subvariety of dimension $n$ which is regular in codimension 1, where $2 \leq n \leq N-2.$  Then for a general finite linear projection $\pi : X \to  \mathbb{P}^{n}$, there are affine open sets $V_{1}, V_{2} \subseteq \mathbb{P}^{n}$ satisfying the following conditions:
\begin{itemize}
\item[{\rm (i)}]{$\pi_{\ast}\cO_{X}$ is monogenic on $V_{1}$ and $V_{2}.$}
\item[{\rm (ii)}]{$V_{1} \cup V_{2}$ contains a line $\ell$ such that $X_{\ell} := \pi^{-1}(\ell)$ is smooth and contained in the regular locus $X^{{\rm reg}}.$}
\end{itemize}
\end{lem}

\begin{proof}
Consider the space of triples $P \subset X \times \PGr(\P^N,N-n-1) \times\PGr(\P^N,N-n)$ defined as the closure of
\[ P^o = \{ (x,\Lambda',\Lambda) : x \in X^{{\rm reg}}, \, x \in \Lambda, \, \Lambda' \subset \Lambda, \, \dim( T_x X \cap T_x\Lambda ) > 1 \}. \]
Let $P' \subset X \times \PGr(\P^N,N-n-1)$ be the image of $P$ under projection.  For a general $(x,\Lambda')$ in $P'$, we have $x \notin \Lambda'$ and therefore $\Lambda$ is the projective span of $\Lambda'$ and $x$.  So $\dim(P) = \dim(P')$.  Let $Q$ be the image of $P$ in $X \times \PGr(\P^N,N-n)$.  Then $P \to Q$ is generically a projective space bundle whose fibers have dimension $N-n$.  So $\dim(P) = \dim(Q) + N - n$.  We will compute the dimension of $Q$, using the projection to $X$.  Let $x \in X^{\rm reg}$.  Then the fiber of $Q$ over $x$ is birationally isomorphic to the set of pairs 
\[ \{ (\alpha, \Lambda) : \alpha \in \Gr(T_x X,2), x \in \Lambda \in \PGr(\P^N,N-n), \alpha \subset T_x\Lambda \}.\]
This set of pairs is a $\Gr(N-2,N-n-2)-$bundle over $\Gr(T_x X,2) \cong \Gr(n,2)$.  So it has dimension $2(n-2) + n(N-n-2)$.  Hence we see that $\dim(Q) = n + 2(n-2) + n(N-n-2)$.  Finally we deduce that 
\[ \dim(P') = N + 2(n-2) + n(N-n-2). \]
In what follows, we denote the linear span of $\Lambda' \in \PGr(\P^N,N-n-1)$ and $x \in X$ by $\Lambda'_x.$  If $P' \to \PGr(\P^N,N-n-1)$ is dominant, then a general fiber has dimension
\[ N + 2(n-2) + n(N-n-2) - (N-n)(n+1) = n- 4. \]

This means that for a general $(N-n-1)-$plane $\Lambda'$, we have $\dim( T_x\Lambda'_x \cap T_x X) \leq 1$ for all $x$ away from a subset of $X^{\rm reg}$ having codimension at least $4$.  If $P' \to \PGr(\P^N,N-n-1)$ is not dominant, then a general $(N-n-1)-$plane $\Lambda'$ would have the property that for any $x \in X^{\rm reg}$, $\dim( T_x\Lambda'_x \cap T_x X) \leq 1$.  In either case, for a general $(N-n-1)-$plane $\Lambda'$, we have that $\Lambda' \cap X = \varnothing$, $\Lambda'_x$ is transverse to $X$ at a general point, and away from a locus of codimension at least two, $T_x\Lambda'_x \cap T_x X$ is at most one dimensional.  Fixing such a $\Lambda'$ for the rest of the proof, we define $Z$ to be the union of $X^{\rm sing}$ and the set of all $x \in X$ for which at least one of these properties fails.  Our discussion thus far implies that $Z$ is of codimension at least 2 in $X.$

Let $\pi:X \to \P^n$ be the finite projection associated to $\Lambda'$, and let $\cA := \pi_*\cO_X$ be the associated locally free sheaf of $\cO_{\P^n}$-algebras.  Given $p \in \P^n \setminus \pi(Z)$ and $x \in \pi^{-1}(p),$ we see that since $T_x\Lambda'_x \cap T_x X$ is at most one-dimensional, the cotangent space to $\pi^{-1}(p)$ at $x$ is at most one-dimensional.  Hence $\cA|_p$ is a monogenic $\cO_{\P^{n},p}-$algebra for all $p \in \P^n \setminus \pi(Z)$ (this will be used shortly). 

Consider an affine open set $V \subset \P^n \setminus \pi(Z)$.    Let $y \in V$ be some point and let $z \in \cA(V)$ be an element such that $z|_y$ is a generator for $\cA|_y$.  Then there is a polynomial $p(z)$ (the characteristic polynomial of $z$) such that the map $\cO_V[z]/\langle p(z) \rangle \to \cA|_V$ is an isomorphism away from a divisor $D \subset V$.  Put $V_1 = V \setminus D$.  Note that $V_1$ is affine and $\cA$ is monogenic on $V_1$.  

Let $\ell \subset \P^n$ be a line which avoids $\pi(Z)$, has nonempty intersection with $V_1$, and is such that $\pi^{-1}(\ell)$ is smooth.  Let $y_1,\dotsc,y_r$ be the members of $\ell \cap (\P^n \setminus V_1)$, and let $V' \subset \P^n\setminus \pi(Z)$ be an affine open set that contains all of the $y_i$.  For each $i$, the algebra $\cA|_{y_i}$ is monogenic, so we can fix a generator $z_i$ of $\cA|_{y_i}$.  Since $\cA(V') \to \prod_{i=1}^r \cA|_{y_i}$ is surjective, there is an element $z \in \cA(V')$ whose restriction to $y_i$ is $z_i$.  Now as before there is a polynomial $q(z)$ such that the map $\cO_{V'}[z]/\langle q(z) \rangle \to \cA|_V$ is an isomorphism away from a divisor $D'\subset V'$.  By construction $y_1,\dotsc,y_r \notin D'$.  Put $V_2 = V' \setminus D'$.  Then $\cA$ is monogenic on $V_2,$ and moreover $\ell \subset V_1 \cup V_2$.
\end{proof}

\begin{prop}
\label{prop:acm-fil-ps}
Let $X \subseteq \mathbb{P}^{N}$ be a normal ACM variety of dimension $n \geq 2,$ and let $\pi : X \to \mathbb{P}^{n}$ be a finite linear projection.  If $\ell \subseteq \mathbb{P}^{n}$ is a line, there exists a filtered vector space $W$ and a characteristic morphism $\phi : \cA \to {\rm End}(W \otimes \cO_{\mathbb{P}^{n}})$ such that $\phi|_{\ell}$ is a filtered pseudomorphism.
\end{prop}

\begin{proof}
Let $x,y,z_2,\dotsc,z_n$ be a coordinate system on $\P^n$ such that $\ell = V(z_2,\dotsc,z_n)$.  It is convenient to work with graded rings instead of schemes.  So let us view $\P^n = \proj(R)$ where $R = \bk[x,y,z_2,\dotsc,z_n]$ and $X = \proj(S)$ where $S$ is a graded Cohen-Macaulay $R$-algebra.  Given that $S$ is a free $R-$module, we fix a homogeneous basis $1=\gamma_1,\dotsc,\gamma_d$ for $S$ as an $R$-module.  Note that $\deg(\gamma_i) > 0$ for $i > 1$.  Moreover $\ell = \proj(\bk[x,y])$.  Let $\Sbr = S/(z_2,\dotsc,z_n)S$ and write $\chi(t)$ and $\chi_\ell(t)$ for the characteristic polynomials of $S$ over $R$ and $\Sbr$ over $\bk[x,y]$, respectively.

As in Lemma \ref{lem:integral-curve-map} we can find a graded $\chi_\ell(t)$-Roby module 
\[ \phi_\ell:\Sbr \oplus \bk[x,y]\cdot T \to \End(W) \tensor \bk[x,y] \]
such that $C_{\phi_\ell}$ is a morphism.  Recall that $\phi_\ell$ must have the property that 
\[
\phi_\ell(\alpha_1\gamma_1 + \dotsc + \alpha_d \gamma_d+\tau T)^d = \chi_\ell( \alpha_1\gamma_1 + \dotsc + \alpha_d \gamma_d + \tau t )\cdot \id_W
\]
where $\alpha_\bt,\tau \in \bk[x,y]$.  Since $\bk[x,y]$ is naturally a subring of $R$, we may define $\phi_0 := \phi_\ell \tensor_{\bk[x,y]} R$.  Write $\chi_0(t)$ for $\chi_\ell(t)$ viewed as an element of $\Sym^\bt_R(S^\vee)[t]$.   Then $\phi_0$ is a graded $\chi_0(t)$-Roby module.

Now we note that $\chi(t) - \chi_0(t) \in (z_\bt)\Sym^\bt_R(S^\vee)[t]$.  So let us write
\[ \chi(t) = \chi_0(t) + \sum_{i=0}^{d-1}\sum_{j=1}^{n_i}{t^i (c_{i,j,1}\Gamma_{k(i,j,1)})(c_{i,j,2} \Gamma_{k(i,j,2)})\dotsm (c_{i,j,d-i}\Gamma_{k(i,j,d-i)})} \]
where $\Gamma_1,\dotsc,\Gamma_d$ are the variables dual to the basis $\gamma_1,\dotsc,\gamma_d$ and $c_{i,j,s} \in R$ has degree equal to $\deg(\gamma_{k(i,j,s)})$.  Now put $m_{i,j} = t^i(c_{i,j,1}\Gamma_{k(i,j,1)})\dotsc (c_{i,j,d-i}\Gamma_{k(i,j,d-i)})$.

We recall the construction of Example \ref{ex:monomial-Roby-module}.  Define a map $\phi_{m_{i,j}}:S \oplus R\{T\} \to \End(R \tensor \bk^d)$ by
\[ \phi_{m_{i,j}}(\gamma_p)(\epsilon_r) = c_{i,j,r-i}\delta^p_{k(i,j,r-i)}\epsilon_{r+1}, \quad (i < r \leq d), \quad \phi_{m_{i,j}}(T)(\epsilon_r) = \begin{cases} \epsilon_{r+1} & 1 \leq r \leq i, \\ 0 & i < r \leq d \end{cases} \]
where $\epsilon_\bt$ are the standard basis vectors of $\bk^d$ and addition in the subscripts are modulo $d$.

Now consider $\phi = \phi_0 \widehat{\tensor}_\xi \phi_{m_{0,1}} \widehat{\tensor}_\xi \dotsm \widehat{\tensor}_\xi \phi_{m_{d,n_d}}$ which is a $\chi_\cA-$Roby action on $\Wt = (\bk^d)^{\tensor n_0+\dotsm+n_{d-1}} \tensor W \tensor R$.  By Proposition \ref{prop:twisted-tensor-roby}, this is a $\chi$-Roby module.  For each $i,j$ at least one of the $c_{i,j,r}$ must be in the ideal $(z_\bt)$.  After possibly reindexing, we may assume that $c_{i,j,d-i} \in (z_\bt)$.  

Consider the filtration $F^i\bk^d = \bk\{\epsilon_i,\dotsc,\epsilon_d\}$ on each ``monomial'' Roby module.  Then upon restriction to $\ell$, the Roby-action of $\Sbr \oplus \bk[x,y]\{T\}$ on $\bk^{\times d} \tensor \bk[x,y]$ via $\phi_{m_{i,j}}$ preserves $F^\bt$.  Moreover, if we put $F^{d+1}=0,$ we have that for each $i,$
\[ (\Sbr \oplus \bk[x,y]\{T\}) \cdot F^i \subset F^{i+1} \]
Let us equip $\Wt$ with the filtration $\widehat{F}^{\bullet}$ which is the tensor product of the filtrations above on its monomial factors and the trivial filtration on $W \tensor R$.  Then the action of $\Sbr \oplus \bk[x,y]\{T\}$ preserves $\widehat{F}^{\bullet}.$  The formation of the $\Z/d\Z-$graded tensor product is bi-functorial on Roby modules.  Since each $m_{i,j}$ vanishes in $\bk[x,y,t]$, we see that the minimal subquotients of $\widehat{F}^{\bullet}$ are simply the $\Z/d\Z-$graded tensor product of $\phi_\ell$ with a number of copies of the rank-one $0$-Roby module correpsonding to the zero map $\Sbr \oplus \bk[x,y]\{T\} \to \End(\bk[x,y])$.  Therefore the minimal subquotients are isomorphic to $\phi_\ell$.

Finally, the formation of $C_\phi$ is functorial.  So the action of $\Sbr$ via $C_\phi$ preserves the tensor product filtration.  Since the associated graded parts are isomorphic to $\phi_{\ell}$ and $C_{\phi_\ell}$ is a morphism, we find that $C_\phi|_\ell$ is a filtered pseudo-morphism.
\end{proof}

\begin{lem}
\label{lem:algebrize}
Let $\cA$ be an ACM sheaf of algebras on $\P^N$ and $\ell \subset \P^N$ a line.  Let $\cAh$ be the formal completion of $\cA$ along $\ell$ and let $\cEh$ be a coherent sheaf of $\cAh$-modules.  Then there is a coherent sheaf $\cE$ of $\cA$-modules whose completion is isomorphic to $\cEh$.
\end{lem}
\begin{proof}
Let $\cI$ be the ideal defining $\ell$.  Consider the sheaf of algebras $\cS = \oplus_{m \geq 0} \cI^m / \cI^{m+1}$ and the coherent graded $\cS$ module $\cF = \oplus_{m \geq 0} \cI^m \cEh  / \cI^{m+1} \cEh$.  Note that $\cS(1)$ is ample on $\spec(\cS)$, where $\cS(1)$ is the pullback of $\cO_\ell(1)$ under the natural map $\spec(\cS) \to \ell$.  It follows that for some $n_0 \gg 0$, $\rmH^1(\cF(n_0)) = 0$.  Observe that 
\[ \rmH^1(\spec(\cS), \cF(n_0) ) = \oplus_m \rmH^1(\ell, (\cI^m \cEh / \cI^{m+1} \cEh)(n_0)), \]
since $\spec(\cS) \to \ell$ is affine.  Therefore, for each $m$,
\[ \rmH^1( (\cI^m \cEh / \cI^{m+1} \cEh)(n_0)) = 0. \]
It follows that the maps
\[ \rmH^0( (\cEh / \cI^{m+1} \cEh)(n_0)) \to \rmH^0((\cEh / \cI^m \cEh)(n_0) ) \]
are surjective.  Therefore the map $\rmH^0( \cEh(n_0) ) \to \rmH^0( (\cEh / \cI \cEh )(n_0) )$ is surjective.  If $V \subset \rmH^0( ( \cEh / \cI \cEh)(n_0))$ is a finite dimensional space of sections such that 
\[ V \tensor \cA(-n_0) \to \cEh / \cI \cEh \]
is surjective and $V' \subset \rmH^0(\cEh(n_0))$ is a lift then 
\[ V' \tensor \cAh(-n_0) \to \cEh \]
is also surjective.  Indeed, the support of the cokernel is empty.  Iterating this arugment we obtain a presentation
\[ W \tensor \cAh(-n_1) \stackrel{\hat{\alpha}}{\to} V' \tensor \cAh(-n_0) \to \cEh \to 0. \]
Now consider the map $\rmH^0( \cA(k) ) \to \rmH^0( \cAh(k) )$.  We wish to show that it is surjective.  Since $\cA$ is dissoci\'{e}, it suffices to show that the maps $\rmH^0(\cO(k)) \to \rmH^0(\cOh(k))$ are surjective for all $k \gg 0$.  If $m > k$ then $\rmH^0(\cO(k)) \to \rmH^0((\cO/\cI^m)(k))$ is an isomorphism.  Hence $\rmH^0(\cA(k)) \to \rmH^0(\cAh(k))$ is an isomorphism.  Therefore there is a morphism
\[ \alpha: W \tensor \cA(-n_1) \to V' \tensor \cA(-n_0) \]
whose completion is $\hat{\alpha}$.  Thus we may take $\cE = \coker(\alpha)$.
\end{proof}

The next result completes the proof of Theorem \ref{main-theorem}.

\begin{thm}
\label{thm:reflexive-delta-ulrich}
Under the hypothesis of Proposition \ref{prop:acm-fil-ps}, there exists a 1-dimensional linear section $C \subseteq X$ and a reflexive sheaf $\cE$ on $X$ such that $\cE|_{C}$ is Ulrich.
\end{thm}

\begin{proof}
Let $\pi :X \to \P^n$ and $V_1,V_2$ be as in Lemma \ref{lem:monogenic-opens} and let $\ell \subset V_1 \cup V_2$ be a line.  Next, let $(W,F^\bt)$ be a filtered vector space and $\phi:\cA = {\pi}_*\cO_X \to \End(W\tensor\cO)$ be as in Proposition \ref{prop:acm-fil-ps} with respect to the line $\ell$.  Let $z_i \in \cA(V_i)$ be an algebra generator for $\cA$ over $V_i$.  Consider the algebra map $\phi_i$ defined by the diagram
\[ \xymatrix{ \cA_i = \cA_{V_i} & \ar[l]_(.6){\cong} \cO_{V_i}[z_i]/(p_i(z_i)) \ar[r] & \End(W \tensor\cO_{V_i}) } \]
where the second map is the unique algebra homomorphism which sends $z_i$ to $\phi(z_i)$.  Write $\cE_i$ for $W\tensor \cO_{V_i}$ with the $\cA_{V_i}$-module structure coming from $\phi_i$.  Note that $\cE_i$ is maximal Cohen-Macaulay over $V_i$ and therefore locally free on $V_i \setminus p(\sing(X))$; in particular, $\cE_i$ is locally free in a neighborhood of $\ell \cap V_i$.  For $i,j=1,2,~ i \neq j,$ we denote the restriction of the $\cA_{V_i}-$module $\cE_i$ to $V_{ij} = V_i \cap V_j$ by $\cE_{ij}.$ 

Let $F^\bt$ be the filtration on $W$.  By assumption, the pseudomorphism $\phi:\cA|_\ell \to \End(W\tensor \cO_\ell)$ induces an $\cO_{\ell}-$algebra morphism $\cA|_\ell \to \prod \End(F^{i+1}W/F^i W \tensor \cO_\ell)$.  Since $X_\ell$ is smooth, the $\cA|_\ell$-module structure on $F^{i+1} W / F^i W \tensor \cO_\ell$ is locally free.  Now, $V_1 \cap V_2 \cap \ell$ is affine.  This means that the filtration $F^\bt \cE_{ij}$ has projective subquotients.  So there is an isomorphism $\gr_{F^\bt} \cE_{ij} \to \cE_{ij}$ which is compatible with the filtration when $\gr_{F^\bt} \cE_{ij}$ is filtered by $F^k = \oplus_{k' \leq k} F^{k'} \cE_{ij} / F^{k'-1}\cE_{ij}$ and which induces the identity on subquotients.  Using these isomorphisms we produce a filtered $\cA_{12}$-module isomorphism
\[ \psi_\ell:\cE_{12}|_\ell \to \cE_{21}|_\ell \]
which induces the same isomorphism on associated graded modules as the identification $\cE_{12} = W \tensor \cO_{\ell \cap V_{12}} = \cE_{21}$.  Let $\cF$ be the vector bundle on $\ell$ obtained by gluing $\cE_1|_\ell$ to $\cE_2|_\ell$ along $\psi_\ell$.  Now since $\psi_\ell$ is filtered, $\cF$ is filtered.  By construction, the subquotients of $\cF$ for this filtration are the same as the subquotients for $W \tensor \cO_\ell$.  Hence the associated graded of $\cF$ is trivial.  It follows that $\cF$ is itself a trivial vector bundle.

Let $\Uh$ be the formal neighborhood of $\ell$ in $\P^n$.  Let $j:\Uh \to \P^n$ and put $\Vh_i = j^{-1}(V_i)$ and $\cAh = j^*\cA$.  Write $\cEh_i = j^*\cE_i$.  Since $\cE_i$ is a locally free $\cA_i$ module on a neighborhood of $\ell \cap V_i$ we see that $\cEh_i$ is a locally free $\cAh$-module.  Since $\P^n$ is separated, $V_{12}=V_1 \cap V_2$ is affine.  Hence $\Vh_{12}=\Vh_1 \cap \Vh_2$ is an affine formal scheme.  Hence the $\cEh_i$ are projective $\cAh_{V_i}$-modules.  Therefore the isomorphism $\psi_\ell$ lifts to an isomorphism
\[ \psi:\cEh_{12} \to \cEh_{21} \]
of $\cAh_{12}$-modules.   The isomorphism $\psi$ gives gluing data for gluing $\cEh_1$ to $\cEh_2$.  Let $\cEh$ be $\cAh$-module obtained by gluing $\cEh_1$ to $\cEh_2$ along $\psi$.  By Lemma \ref{lem:algebrize}, there is a sheaf $\cE$ of $\cA$-modules whose restriction to $\Uh$ is $\cEh$.  Since $\cEh$ is locally free, $\cE$ is locally free in a neighborhood of $\ell$.  So replacing $\cE^{\vv}$ is also isomorphic to $\cEh$ on $\Uh$.  Hence $\cE^{\vv}$ is the desired sheaf.
\end{proof}


\section{Generalities on $\delta-$Ulrich Sheaves} \label{sec:generalities}
In this section, $X \subseteq \mathbb{P}^{N}$ is a normal ACM variety of degree $d$ and dimension $n \geq 2$. 

\begin{lem}
\label{lem:rest-lin-sec}
Suppose that $\cE$ is a locally CM sheaf on $X$ whose restriction to a
linear section $Y$ of dimension at least 2 is Ulrich.  Then $\cE$ is Ulrich.
\end{lem}

\begin{proof}
Let $\pi : X \to \mathbb{P}^{n}$ be a finite linear projection.  Since $\cE$ is a locally CM sheaf on $X,$ the direct image $\pi_{\ast}\cE$ is a locally CM sheaf on a smooth variety, and is therefore locally free.  Replacing $\cE$ by $\pi_{\ast}\cE$ if necessary, we can assume without loss of generality that $\cE$ is locally free and $X=\mathbb{P}^{n}.$  We will show that $\cE$ is trivial.    

By induction on dimension we may assume that $\dim(Y) = n-1$ so that $Y$ is a hyperplane.  Our hypothesis on $\cE$ amounts to $\cE|_{Y}$ being a trivial bundle.  To show that $\cE$ is trivial, it is enough to check that $h^{0}(\cE)={\rm rk}(\cE).$  (See Proposition \ref{prop:ulrich-equiv-conditions}.)  We will do this by showing that restriction map $H^{0}(\cE) \to H^{0}(\cE|_Y)$ is surjective, since $h^{0}(\cE|_{Y})={\rm rk}(\cE).$

For each positive integer $j,$ let $jY$ the $(j-1)-$st order thickening of $Y.$  We claim that for all $m \geq 1,$ the restriction map $H^0(\cE|_{(m+1)Y}) \to H^0(\cE|_{mY})$ is an isomorphism.  Grant this for the time being.  If we fix $m_{0} \gg 0$ for which $H^{1}(\cE(-m_{0}))=0,$ then the restriction map $H^{0}(\cE) \to H^{0}(\cE|_{m_{0}Y})$ is surjective, and the claim yields the desired surjectivity of the map $H^{0}(\cE) \to H^{0}(\cE|_{Y}).$ 

Turning to the proof of this claim, an obvious snake lemma argument gives the following exact sequence for each $m \geq 1:$
\[ 0 \to \cE|_Y(-m) \to \cE|_{(m+1)Y} \to \cE|_{mY} \to 0. \]
Since $\cE|_Y$ is trivial and $\dim(Y) > 1,$ $h^{0}(\cE|_{Y}(-m))=h^1(\cE|_Y(-m)) = 0$, so the map $H^0(\cE|_{(m+1)Y}) \to H^0(\cE|_{mY})$ is an isomorphism.

\end{proof}

If $\cE$ is a $\delta-$Ulrich sheaf on $X$ then for a general 1-dimensional linear section $Y \subset X$, $\cE|_Y$ is Ulrich.  Indeed, if we consider a finite linear projection $\pi:X \to \P^n$ then $\pi_*\cE$ is a reflexive sheaf whose restriction to a given line is trivial.  Since trivial sheaves on $\P^1$ are rigid, the restriction of $\pi_*\cE$ to nearby lines is also trivial.  We also point out that if $X_H$ is a general hyperplane section of $X$ then $\cE|_{X_H}$ is $\delta$-Ulrich on $X_H$.

\begin{prop}
\label{prop:ulrich-equiv-conditions}
Let $\cE$ be a $\delta-$Ulrich sheaf of rank $r$ on $X.$  Then the following are equivalent.
\begin{itemize}
\item[(i)]{$\cE$ is Ulrich.}
\item[(ii)]{$h^{0}(\cE)=dr.$}
\end{itemize}
\end{prop}

\begin{proof}
The implication $(i) \Rightarrow (ii)$ is clear, so we will focus on proving $(ii) \Rightarrow (i).$  Assume $h^{0}(\cE)=dr.$  If $\pi : X \to \P^n$ is a finite linear projection, then $h^{0}(\cE) = h^{0}(\pi_{\ast}\cE),$ and $\cE$ is Ulrich if and only if $\pi_{\ast}\cE$ is Ulrich with respect to $\cO_{\P^{n}}(1),$ i.e. a trivial vector bundle.  We may then assume without loss of generality that $(X,\cO(1))=(\P^{n},\cO_{\P^{n}}(1))$ (in particular, $d=1$).  In this case the $\delta-$Ulrich condition on $\cE$ is equivalent to $\cE|_{\ell} \cong \cO_{\ell}^{r}$ for some line $\ell \subseteq \mathbb{P}^{n},$ and the Ulrich condition on $\cE$ is equivalent to $\cE \cong \cO_{\mathbb{P}^{n}}^{r}.$  

If $h^0(\cE) = r,$ then since $h^0(\cE \tensor \cI_\ell) = 0$ we find that the evaluation map $ev:\rmH^0(\cE) \tensor \cO \to \cE$ restricts to the evaluation map $\rmH^0(\cE|_\ell)\tensor\cO_\ell \to \cE|_\ell$, which is an isomorphism.  If $\cF$ is the cokernel of $ev,$ its support cannot intersect $\ell$ and therefore has codimension at least two.  We then have that $\Ext^1(\cF,\cO^r) \cong (H^{n-1}(\cF(-n-1))^{\ast})^{r} \cong 0.$  Since $\cE$ is reflexive, it is torsion-free, so we must have $\cF = 0,$ i.e. that $ev$ is an isomorphism.
\end{proof}

We shall now consider stability properties of $\delta-$Ulrich bundles.

\begin{lem}
\label{lem:psu-ss}
Let $\cE$ be a $\delta$-Ulrich sheaf on $X$.  Then $\cE$ is $\mu-$semistable, and $\omega_{X} \otimes \cE^{\vee}(n+1)$ is also $\delta$-Ulrich.
\end{lem}

\begin{proof}
Let $\cF$ be a torsion-free quotient of $\cE,$ and let $Y \subset X$ be a smooth 1-dimensional linear section such that $\cE|_{Y}$ is Ulrich; we may also assume $Y$ avoids the singular loci of $\cE$ and $\cF.$  Then $\cF|_{Y}$ is a torsion-free quotient of the semistable bundle $\cE|_{Y};$ consequently $\mu(\cE)=\mu(\cE|_{Y}) \leq \mu(\cF_{Y}) = \mu(\cF).$

The second part of the statement follows from the adjunction formula and the fact that if $C$ is a curve embedded in projective space by $\cO_{C}(1)$ and $\cE'$ is an Ulrich bundle on $C,$ then $\omega_{C} \otimes \cE'^{\vee}(2)$ is also Ulrich.
\end{proof}

\begin{lem}
\label{lem:jh-filtration}
If $\cE$ is a $\delta$-Ulrich sheaf on $X$ which is strictly $\mu-$semistable, then there exists a $\mu-$stable subsheaf $\cE' \subset \cE$ which is $\delta-$Ulrich.
\end{lem}

\begin{proof}
Let $\cE' \subset \cE$ be the maximal destabilizing subsheaf of $\cE.$  We will show that $\cE'$ is $\delta-$Ulrich.  Let $Y \subseteq X$ be a general 1-dimensional linear section of $\cE$ which avoids the singular loci of $\cE'$ and $\cE.$  Then $\cE'|_{Y}(-1)$ is a subsheaf of $\cE|_{Y}(-1).$  Since the latter is an Ulrich sheaf, we have that $h^{0}(\cE|_{Y}(-1))=0,$ and it follows that $h^{0}(\cE'|_{Y}(-1))=0$ as well.  The slope of $\cE'|_{Y}(-1)$ is equal to that of $\cE|_{Y}(-1),$ so Riemann-Roch implies that $h^{0}(\cE'|_{Y}(-1))=h^{1}(\cE'|_{Y}(-1))=0;$ therefore $\cE'|_{Y}$ is Ulrich.
\end{proof}

\begin{lem}
\label{lem:down-twists}
Let $\cE$ be a $\delta$-Ulrich sheaf on $X.$  Then for all $k \geq 1,$ we have $h^{0}(\cE(-k))=0.$
\end{lem}

\begin{proof}
We proceed by induction on ${\rm dim}(X).$  Let $X_{H} \subset X$ be a general hyperplane section.  Then for each $k \geq 1$ we have the exact sequence
\[ 0 \to \cE(-k-1) \to \cE(-k) \to \cE|_{X_H}(-k) \to 0 \]
Since negative twists of an Ulrich sheaf have no global sections, our inductive hypothesis implies $h^{0}(\cE_{X_H}(-k))=0$; it follows that $h^{0}(\cE(-k)) = h^{0}(\cE(-1))$ for all $k \geq 1.$  We need only exhibit some $k' \geq 1$ such that $h^{0}(\cE(-k'))=0.$  Since $\cE$ and all its twists are $\mu$-semistable by Lemma \ref{lem:psu-ss}, any positive $k' > \mu(\cE)$ will do.
\end{proof}

\begin{rmk}
\label{rmk:lin-det-examples}
We exhibit for each $n \geq 2$ a smooth ACM variety of dimension $n$ admitting $\delta-$Ulrich sheaves which are not Ulrich.  Consider the Segre variety $X := \P^1 \times \P^{n-1} \subseteq \P^{2n-1},$ and let $H$ be the hyperplane class of $X$.  Recall that $X$ is cut out in $\mathbb{P}^{2n+1}$ by the maximal minors of the generic $2 \times n$ matrix of linear forms.  It follows from Proposition 2.8 of \cite{BHU} that the degeneracy locus $D \subseteq X$ of the first row of this matrix is a divisor whose associated line bundle $\cO_{X}(D)$ is an Ulrich line bundle on $X.$  The general 1-dimensional linear section $X' \subset X$ is a rational normal curve of degree $n,$ so if $\cL \in {\rm Pic}(X)$ satisfies $H^{n-1} \cdot \cL = 0,$ the restriction $\cL|_{X'}$ is the trivial bundle; in particular $\cL(D)$ is $\delta-$Ulrich.  Since the set $(H^{n-1})^{\perp}$ of all such $\cL$ is a corank-1 subgroup of ${\rm Pic}(X),$ we can choose $\cL \in (H^{n-1})^{\perp}$ such that $\cL(D)$ lies outside the effective cone of $X,$ e.g. satisfies $H^{0}(\cL(D))=0.$  In this case $\cL(D)$ is not Ulrich.
\end{rmk}


\section{The surface case} \label{sec:surface}
Throughout this section we consider a normal surface $X$ with a very ample line bundle $\cO_X(1)$.  We assume that $X$ has a $\delta$-Ulrich sheaf $\cE$, but not necessarily that $X$ is ACM.  

\subsection{Relation to Instanton Sheaves}
\label{subsec:instanton}

\begin{prop}
\label{prop:surface-instanton}
Let $\cE$ be a $\delta-$Ulrich sheaf of rank $r$ on $X,$ and let $\pi : X \to \P^{2}$ be a finite linear projection.  Then $\pi_{\ast}\cE$ is $\mu-$semistable, and it is an instanton sheaf on $\P^2,$ i.e. the cohomology of a monad of the form
\begin{equation}\label{eq:linear-monad} 
0 \to \cO_{\P^2}(-1)^{\oplus m} \to \cO_{\mathbb{P}^{2}}^{\oplus rd+2m} \to \cO_{\P^2}(1)^{\oplus m} \to 0
\end{equation}
where $d = \deg(X)$ and $m = h^1(\cE(-1))$.
\end{prop}

\begin{proof}
If $\cE$ is $\delta-$Ulrich, then $\pi_{\ast}\cE$ is a reflexive, and thus locally free, sheaf on $\P^{2}.$  Since the restriction of $\pi_{\ast}\cE$ to a general line is trivial, $\pi_{\ast}\cE$ is $\mu-$semistable of degree 0, and given that $h^{1}(\cE(-1))=h^{1}(\pi_{\ast}\cE(-1))$, Theorem 17 of \cite{Jar} implies our result.  
\end{proof}

The following statement can be obtained from a short elementary argument, but it seems appropriately stated as a consequence of Proposition \ref{prop:surface-instanton}.

\begin{cor}
\label{cor:delta-charge}
A $\delta-$Ulrich sheaf $\cE$ on $X$ is Ulrich if and only if $H^{1}(\cE(-1))=0.$ \hfill \qedsymbol
\end{cor}

At this point it is natural to ask if, given a $\delta-$Ulrich sheaf $\cE$ on $X,$ there is a $\delta-$Ulrich sheaf $\cE'$ on $X$ with $h^{1}(\cE'(-1)) < h^{1}(\cE(-1))$; an affirmative answer combined with Theorem \ref{main-theorem} would imply that every normal ACM surface admits an Ulrich sheaf.  The next result shows that it is enough to consider stable $\delta-$Ulrich bundles.

\begin{lem}
\label{lem:halving}
Let $X$ be a smooth ACM surface, and let $\cE$ be a $\delta$-Ulrich sheaf on $X$ which is strictly $\mu$-semistable with $h^{1}(\cE(-1)) = m.$  Then $X$ admits a locally free $\delta$-Ulrich sheaf $\cE'$ with ${\rm rk}(\cE') < {\rm rk}(\cE)$ and $h^{1}(\cE'(-1)) \leq \frac{m}{2}.$
\end{lem}

\begin{proof}
A $\mu$-Jordan-H\"{o}lder filtration of $\cE$ yields an exact sequence
\[ 0 \to \cF(-1) \to \cE(-1) \to \cG(-1) \to 0 \]
where $\cF$ and $\cG$ are both $\delta$-Ulrich sheaves and $\cF$ is both $\mu-$stable and locally free.  Since $h^{2}(\cF(-1))=h^{0}(\omega_{X} \otimes \cF^{\vee}(1)),$ and $\omega_{X} \otimes \cF^{\vee}(3)$ is $\delta$-Ulrich by Lemma \ref{lem:psu-ss}, we have from Lemma \ref{lem:down-twists} that $h^{2}(\cF(-1))=0.$  Another application of this Lemma implies that $h^{0}(\cG(-1))=0$.  We may then conclude that  $\min\{h^{1}(\cF(-1)),h^{1}(\cG(-1))\} \leq \frac{m}{2},$ and the result follows by taking the reflexive hull of $\cG$ if necessary.
\end{proof}

\begin{rmk}
Suppose that $X \subseteq \P^N$ is an ACM variety of dimension $n$ and let $\pi:X \to \P^n$ be a finite linear projection.  Jardim \cite{Jar} defines a notion of instanton sheaves on $\P^n$ for any $n > 1$; according to his definition, a $\mu-$semistable reflexive instanton sheaf is $\delta-$Ulrich.  However, only in the case $n = 2$ is the direct image $\pi_*\cE$ of a $\delta-$Ulrich sheaf on $X$ clearly an instanton sheaf.  For $n > 2$ an instanton sheaf must satisfy additional cohomology-vanishing which does not follow from having trivial restriction to a line.   Our construction does not appear to allow for any control over the cohomology of $\delta-$Ulrich sheaves.
\end{rmk}

\begin{rmk}
\label{rmk:non-loc-free}
Instanton sheaves can be used to show that a smooth projective threefold  $X \subseteq \P^N$ admitting an Ulrich sheaf (e.g. a smooth complete intersection of $N-3$ hypersurfaces, by \cite{HUB}) admits a $\delta-$Ulrich sheaf which is not locally free.  Let $\pi : X \to \P^3$ be a general linear projection, and let $z \in \P^3$ be a point not contained in the branch divisor of $\pi.$  By Example 5 in \cite{Jar} there exists a rank-3 $\mu-$semistable reflexive instanton (and thus $\delta-$Ulrich) sheaf $\cG$ on $\P^3$ whose singular locus is exactly $z.$  If $\cF$ is an Ulrich sheaf on $X,$ then $\cF$ is locally free by the smoothness of $X$, and $\cE := \cF \otimes \pi^{\ast}\cG$ is a reflexive sheaf on $X$ whose singular locus is exactly $\pi^{-1}(z).$  Since $\cG$ is locally free away from $z,$ the sheaves $\pi_{\ast}\cE$ and $\pi_{\ast}\cF \otimes \cG$ are isomorphic away from $z.$  Consequently there exists a line $\ell \subseteq \P^3$ not containing $z$ such that the restricton of $\pi_{\ast}\cE$ to $\ell$ is a trivial bundle.  It follows that $\cE$ is $\delta-$Ulrich. 
\end{rmk}

\subsection{Proof of Theorem B}

For the next two Lemmas, we consider a $\delta$-Ulrich sheaf $\cE$ on $\P^2$ for the canonical polarization $\cO_{\P^2}(1)$.  In general, it is difficult to understand how an abstract $\delta$-Ulrich sheaf will restrict to a curve in $\P^2$.  However, we can say something when the curve is a general smooth conic.
\begin{lem}
\label{lem:rest-to-conic}
Let $C \subset \P^2$ be a general smooth conic.  Then $\cE|_C$ is trivial.  
\end{lem}
\begin{proof}
Any smooth conic is isomorphic to $\P^1,$ so it is enough to show that $\cE|_{C}$ is of degree 0 and semistable when $C$ is a general element of $|\cO_{\P^2}(2)|.$  Since the restriction of $\cE$ to a general line is a trivial bundle, it follows that $\det(\cE)$ is trivial.  Consequently the restriction of $\cE$ to any plane curve has degree 0.  We now turn to semistability.  Consider the universal plane conic
\[ 
\mathcal{C} := \{ (p,C) \in \mathbb{P}^{2} \times |\cO_{\mathbb{P}^{2}}(2)| : p \in C \}
\] 
with its associated projections $p_{1} : \mathcal{C} \to \P^{2}, p_{2} : \mathcal{C} \to |\cO_{\P^{2}}(2)|.$  Our goal amounts to showing that the restriction of $p_{1}^{\ast}\cE$ to the general fiber of $p_{2}$ is semistable.  Given that this is an open condition on the fibers of $p_{2}$ (e.g. Proposition 2.3.1 in \cite{HL}) it is enough to check the semistability of $\cE|_{C_0}$ when $C_{0} = L \cup L'$ for distinct lines $L,L' \subseteq \P^{2}$ satisfying the property that $\cE|_{L}$ and $\cE|_{L'}$ are trivial.  If we twist the Mayer-Vietoris sequence
\[ 0 \to \cO_{C_0} \to \cO_{L} \oplus \cO_{L'} \to \cO_{L \cap L'} \to 0 \] 
by $\cE$ and take cohomology, we see that the induced difference map $H^{0}(\cE|_{L}) \oplus H^{0}(\cE|_{L'}) \to H^{0}(\cE|_{L \cap L'})$ is surjective.  Therefore $\cE|_{C_0}$ is locally free of rank ${\rm rk}(\cE)$ with ${\rm rk}(\cE)$ global sections, i.e. $\cE|_{C_0} \cong \cO_{C_0}^{\oplus {\rm rk}(\cE)}.$  In particular, $\cE|_{C_0}$ is semistable.
\end{proof}

\begin{lem}\label{lem:ec-formula}
Let $\cF$ be an $\cO(2)$-Ulrich sheaf on $\P^2$.  Then $\cE \tensor \cF$ is $\delta$-Ulrich for $\cO(2)$ and we have
\[ \chi(\cE \tensor \cF) = \rk(\cF)\left( \chi(\cE) + 3 \rk(\cE) \right). \]
\end{lem}
\begin{proof}
Since the restriction of $\cE$ and $\cF$ to a general conic are trivial and Ulrich, respectively, and Ulrich sheaves are stable under taking direct sums, we see that the restriction of $\cE \tensor \cF$ to a general conic is Ulrich.  Hence $\cE\tensor\cF$ is $\delta-$Ulrich for $\cO(2)$. 

Since $\cE$ is $\delta-$Ulrich for $\cO_{\P^2}(1),$ Proposition \ref{prop:surface-instanton} implies that it is the cohomology of a monad of the form
\[ 
0 \to \cO_{\P^2}(-1)^{m} \to \cO_{\mathbb{P}^{2}}^{\rk(\cE)+2m} \to \cO_{\P^2}(1)^{m} \to 0
\] 
where $m = h^1(\cE(-1))$.  Twisting by $\cF,$ we have that if $\ell \subseteq \mathbb{P}^{2}$ is a line, then
\begin{align*}
\chi(\cE \otimes \cF) & = (\rk(\cE)+2m) \cdot \chi(\cF) -m\cdot(\chi(\cF(-1))+\chi(\cF(1))) \\
& =\rk(\cE) \cdot \chi(\cF) + m \cdot ((\chi(\cF)-\chi(\cF(-1)))-(\chi(\cF(1))-\chi(\cF)))\\
& =\rk(\cE) \cdot \chi(\cF) + m \cdot (\chi(\cF|_{\ell})-\chi(\cF(1)|_{\ell}))\\
& = \rk(\cE) \cdot \chi(\cF) - m \cdot \rk(\cF)\\
\end{align*}
We have from Riemann-Roch that $\chi(\cE) = {\rm ch}_{2}(\cE)+\rk(\cE) = \rk(\cE)-m;$ also, the fact that $\cF$ is Ulrich with respect to $\cO(2)$ implies that $\chi(\cF)=4 \rk(\cF).$  Summarizing, we have that
\[ 
\chi(\cE \otimes \cF) = 4 \rk(\cE) \cdot \rk(\cF) + (\chi(\cE)-\rk(\cE)) \cdot \rk(\cF) = \rk(\cF)\left( \chi(\cE) + 3 \rk(\cE) \right).
\] 
\end{proof}

One way to explain the previous Lemma is that Ulrich sheaves on $\P^2$ for $\cO(2)$ are slightly positive.  (The main example of an $\cO(2)$-Ulrich sheaf is the tangent bundle $T\P^2$.)  So tensoring with such a sheaf should enlarge the space of sections while decreasing the the higher cohomology.  The Lemma makes this intuition precise and the next Theorem uses this idea to produce $\delta$-Ulrich sheaves with sections (after changing the polarization).

\begin{thm}\label{thm:good-sequence}
Assume that $(X,\cO_X(1))$ admits a $\delta$-Ulrich sheaf.  Then there exists a sequence $\cE_m$ of sheaves on $X$ such that $\cE_m$ is $\delta$-Ulrich for $\cO_X(2^m)$ and 
\[ \lim_{m \to \infty} \alpha(\cE_m) = 1. \]
In particular, for $m \gg 0$, $h^0(\cE_m) > 0$.
\end{thm}
\begin{proof}
We will construct the sequence $\cE_m$ inductively as follows.  Put $\cE_0 = \cE$ and fix an $\cO(2)$-Ulrich sheaf $\cF$ on $\P^2$.  Now, assume we have constructed $\cE_0,\dotsm, \cE_m$ such that $\cE_i$ is $\delta$-Ulrich with respect to $\cO_X(2^i)$.  To construct $\cE_{m+1}$ we consider the embedding $X \to \P^N$ determined by $\cO_X(2^m)$.  Let $\pi:X \to \P^2$ be a finite map obtained as the compostion of $i$ with a general linear projection $\P^N \dashrightarrow \P^2$.  Define $\cE_{m+1} = \cE_m \tensor \pi^*\cF$.  By Lemma \ref{lem:ec-formula}, $\pi_\ast( \cE_m \tensor \pi^\ast \cF) = \pi_\ast(\cE_m) \tensor \cF$ is $\delta$-Ulrich for $\cO(2)$ since $\pi_\ast \cE_m$ is $\delta$-Ulrich for $\cO(1)$ and $\cF$ is Ulrich for $\cO(2)$.  Thus $\cE_{m+1}$ is $\delta$-Ulrich for $\pi^\ast\cO(2) = \cO_X(2^{m+1})$.  Moreover,
\[ \chi(\cE_{m+1}) = \rk(\cF)( \chi(\cE_m) + 3 \rk(\pi_\ast(\cE_m)) ) = \rk(\cF)(\chi(\cE_m) + 3 \rk(\cE_m)\deg(\cO_X(2^m))). \]
Since $\deg(\cO_X(2^{m+1})) = 4 \deg(\cO_X(2^m))$, we can write
\[ \frac{ \chi(\cE_{m+1}) }{ \rk(\cE_{m+1}) \deg(\cO_X(2^{m+1})) } = \frac{1}{4} \cdot \frac{ \chi(\cE_m)}{\rk(\cE_m)\deg(\cO_X(2^m))} + \frac{3}{4}. \]
Now it is clear that 
\[ \lim_{m \to \infty} \frac{ \chi(\cE_m) }{ \rk(\cE_m)\deg(\cO_X(2^m)) } = 1. \]
On the other hand we have
\[ \frac{ \chi(\cE_m) }{ \rk(\cE_m)\deg(\cO_X(2^m)) } \leq \alpha(\cE_m,\cO_X(2^m)) \leq 1 \]
and the Theorem follows immediately.
\end{proof}

\begin{rmk}
Suppose that $\cE$ is a $\delta$-Ulrich sheaf for $\cO(1)$ and $\cF$ is an $\cO(2)$-Ulrich sheaf on $\P^2$.  A calculation similar to those in the proof of Lemma \ref{lem:ec-formula} shows that 
\[ h^1(\cE \tensor \cF(-2)) = \rk(\cF) h^1(\cE(-1)) \]
Hence 
\[ \frac{h^1(\cE \tensor \cF(-2)}{\rk(\cE \tensor \cF)} = \frac{h^1(\cE(-1))}{\rk(\cE)}. \]
So while $\cE \tensor \cF$ is closer to being $\cO(2)$-Ulrich than $\cE$ is to being $\cO(1)$-Ulrich as measured by $\alpha(-)$, it is no closer at all by this other measure.
\end{rmk}

\begin{rmk}
The minimum rank of an $\cO(2)$-Ulrich bundle on $\P^2$ is two.  So the ranks of the sheaves $\cE_m$ in Theorem \ref{thm:good-sequence} are growing exponentially.  
\end{rmk}

\subsection{Intermediate Cohomology Modules}

Let $\cE$ be a $\delta$-Ulrich sheaf on $X$.  Our last result describes the structure of the graded module $H^{1}_{\ast}(\cE)$ in a way that refines Corollary \ref{cor:delta-charge}.  First we need a definition.

\begin{definition}
Let $S$ be a standard graded ring and $M$ a finitely generated $S$ module.  We say that $M$ has the \emph{Weak Lefschetz Property} \cite{MN} if there is a linear element $z \in S_1$ such that each multiplication map $\mu_z:M_i \to M_{i+1}$ has maximum rank.
\end{definition}

\begin{prop} 
\label{prop:coh-mod-WLP}
The graded module $\rmH^1_\ast(\cE)$ over the graded ring $S_X = \rmH^0_\ast(\cO_X)$ has the Weak Lefschetz property.  Moreover, the following inequalities hold:
\begin{align*}
& h^1(\cE(i)) \leq h^1(\cE(i+1)) \quad (i \leq -2) \\
& h^1(\cE(i)) \geq h^1(\cE(i+1)) \quad (i \geq -2 ) \\
\end{align*}
\end{prop}
\begin{proof}
Let $H \subset X$ be a hyperplane section (with respect to $\cO_X(1)$) such that $\cE|_H$ is Ulrich and $z \in \rmH^0(\cO_X(1))$ a defining section.  Then consider the long exact sequence
\begin{equation}\label{eq:long-exact-hyperplane}\tag{$\ast$}
\xymatrix{ \rmH^0(\cE|_H(i+1)) \ar[r] & \rmH^1(\cE(i)) \ar[r]_(.45){\mu_z} & \rmH^1(\cE(i+1)) \ar[r] & \rmH^1(\cE|_H(i+1)) }
\end{equation}
on cohomology induced by
\[ 0 \to \cE(i) \to \cE(i+1) \to \cE|_H(i+1) \to 0. \]
Recall that since $\cE|_H$ is Ulrich, we have
\[ \rmH^0(\cE|_H(i+1)) = 0, \, (i \leq -2), \quad \text{and} \quad \rmH^1(\cE|_H(i+1)) = 0, \, (i \geq -2). \]
So if $i < -1$, the map $\mu_z$ in \eqref{eq:long-exact-hyperplane} is injective, and if $i > -3$ it is surjective.  
\end{proof}

\begin{rmk}
An immediate consequence of Proposition \ref{prop:coh-mod-WLP} is that $\rmH^{1}_{\ast}(\cE)$ is generated in degree at most $-2.$  We show this is the best possible statement by exhibiting for each $s \geq 2$ a $\delta-$Ulrich sheaf $\cE_{s}$ such that $\rmH^{1}_{\ast}(\cE_{s})$ has a generator in degree $-s.$  Consider the simplest of the varieties discussed in Remark \ref{rmk:lin-det-examples}, i.e. a smooth quadric surface $X \subseteq \mathbb{P}^{3}.$  Let $L_{1}, L_{2}$ be the line classes which generate ${\rm Pic}(X).$  Then $H=L_{1}+L_{2}$ and $H^{\perp}$ is generated by $L_{1}-L_{2}.$  For each $s \in \mathbb{Z},$ the line bundle $\cE_{s} := \cO_{X}(sL_{1}+(1-s)L_{2})$ is $\delta$-Ulrich, and fails to be Ulrich precisely when $s \neq 0,1.$  For $s \geq 2$ and $k \in \mathbb{Z}$ we have
\[
h^{1}(\cE_{s}(-s+k)) = h^{1}(\cO_{\mathbb{P}^{1}}(k) \boxtimes \cO_{\mathbb{P}^{1}}(1-2s+k)) = \begin{cases} (k+1)(2s-k-2), & 0 \leq k \leq 2s-3 \\ 0 & \mbox{otherwise} \end{cases}
\]
\end{rmk}

\bibliographystyle{alpha}
\bibliography{ulrich}

\end{document}